\documentclass[11pt,reqno]{amsart}
\usepackage{color}

\usepackage{amsfonts,amscd}
\usepackage{amssymb}
\usepackage{url}
\usepackage[english]{babel}

\theoremstyle{plain}
\newtheorem{theorem}                 {Theorem}      [section]
\newtheorem{proposition}  [theorem]  {Proposition}
\newtheorem{corollary}    [theorem]  {Corollary}
\newtheorem{lemma}        [theorem]  {Lemma}

\theoremstyle{definition}
\newtheorem{example}      [theorem]  {Example}
\newtheorem{remark}       [theorem]  {Remark}
\newtheorem{definition}   [theorem]  {Definition}

\numberwithin{equation}{section}

\def \R{{\mathbb R}}
\def \rn{{\mathbb R}}
\def \s{{\mathbb S}}

\def \C{{\mathbb C}}
\def \cn{{\mathbb C}}
\def \hn{{\mathbb H}}
\def \H{{\mathbb H}}

\def \B{\mathcal B}

\def \S{\mathcal S}

\def \nab#1#2{\hbox{$\nabla$\kern -.3em\lower 1.0 ex
\hbox{$#1$}\kern -.1 em {$#2$}}}

\def\Re{\mathfrak R\mathfrak e}

\def \g{\mathfrak{g}}
\def \h{\mathfrak{h}}
\def \k{\mathfrak{k}}

\def \m{\mathfrak{m}}

\def \p{\mathfrak{p}}

\def \un{\mathfrak{u}}

\def \GLR#1{\text{\bf GL}_{#1}(\rn)}

\def \glr#1{\mathfrak{gl}_{#1}(\rn)}
\def \GLC#1{\text{\bf GL}_{#1}(\cn)}
\def \glc#1{\mathfrak{gl}_{#1}(\cn)}
\def \GLH#1{\text{\bf GL}_{#1}(\hn)}

\def \SO#1{\text{\bf SO}(#1)}
\def \so#1{\mathfrak{so}(#1)}

\def \SOO#1#2{\text{\bf SO}(#1,#2)}

\def \U#1{\text{\bf U}(#1)}
\def \u#1{\mathfrak{u}(#1)}

\def \SU#1{\text{\bf SU}(#1)}

\def \Sp#1{\text{\bf Sp}(#1)}
\def \sp#1{\mathfrak{sp}(#1)}

\DeclareMathOperator{\Div}{div}

\DeclareMathOperator{\trace}{trace}

\numberwithin{equation}{section}

\begin{document}

\title[Biharmonic functions]
{Biharmonic functions on the\\ classical compact simple Lie groups}


\author{Sigmundur Gudmundsson, Stefano Montaldo and Andrea Ratto}

\address{Mathematics, Faculty of Science\\ University of Lund\\
Box 118, Lund 221\\
Sweden}
\email{Sigmundur.Gudmundsson@math.lu.se}

\address{Universit\`a degli Studi di Cagliari\\
Dipartimento di Matematica e Informatica\\
Via Ospedale 72\\
09124 Cagliari, Italia}
\email{montaldo@unica.it}

\address{Universit\`a degli Studi di Cagliari\\
Dipartimento di Matematica e Informatica\\
Viale Merello 93\\
09123 Cagliari, Italia}
\email{rattoa@unica.it}

\thanks{Work partially supported by: PRID 2015 -- Universit\`a degli Studi di Cagliari}

\allowdisplaybreaks

\begin{abstract}
The main aim of this work is to construct several new families of proper biharmonic functions defined on open subsets of the classical compact simple Lie groups $\SU n$, $\SO n$ and $\Sp n$. We work in a geometric setting which connects our study with the theory of
submersive harmonic morphisms. We develop a general duality principle and use this to interpret our new examples on the Euclidean sphere $\s ^3$ and on the hyperbolic space $\H^3$.
\end{abstract}



\maketitle

\section{Introduction}

To illustrate the main theme of this paper, let us first consider a collection of rather simple functions $f_n:\C^*\to\C$. Here $n$ is a natural number, $\C^*$ is the punctured plane of non-zero complex numbers and
\begin{equation}\label{examplebis}
f_n(z)=\left ( \frac{z}{\bar{z}} \right )^n.
\end{equation}
Further let $\Delta$ denote the standard Laplace operator in the Euclidean plane $\cn$ and $r$ be a positive integer. Then it is not difficult to prove by induction the following interesting formula:
\begin{equation}\label{magic-formula-plane}
 (\Delta^rf_n)(z)=\left(-\frac{4}{|z|^2}\right)^{r}\left(\prod_{k=1}^r(n^2-(k-1)^2)\right)f_n(z).
\end{equation}
Here $\Delta^r$ is the iterated Laplacian given by $\Delta^r=\Delta(\Delta^{(r-1)})$ and $\Delta^0 f=f$. As a direct consequence of equation \eqref{magic-formula-plane}, we have the following result.

\begin{proposition}\label{proposition-plane}
Let $n$ be a natural number and the function $f_n:\C^*\to\C$ be defined as in \eqref{examplebis}. Then $f_n$ is a proper $r$-harmonic function (i.e. $\Delta^r f_n=0$ and $\Delta^{r-1}f_n\neq 0$) if and only if $r=n+1$.
\end{proposition}


\vskip .2cm

The literature on biharmonic functions is vast, but usually the domains are either surfaces or open subsets of flat Euclidean space.  In this paper we construct the first proper biharmonic functions from open subsets of the classical compact simple Lie groups $\SU n$, $\SO n$ and $\Sp n$, equipped with their standard biinvariant Riemannian metrics.

It is a well-known fact that the matrix coefficients of any finite-dimensional irreducible representation of a compact semi-simple Lie group are eigenfunctions of the corresponding Laplace-Beltrami operator.  For this see \cite{Gud-Sve-Ville-1} or Proposition 5.28 of \cite{Kna}.

In Theorems \ref{teo-quotient-linear}, \ref{theorem-orthogonal-group} and \ref{teo-quotient-linear-quaternionic} we produce a large collection of complex-valued biharmonic functions on open subsets of $\SU n$, $\SO n$ and $\Sp n$ respectively.   They are all quotients of linear combinations of the matrix coefficients for the {\it standard} irreducible representation of the corresponding group. In Theorem \ref{theorem-biharmonicity-2x2-determinants} we then construct solutions which are quotients of $2\times2$ determinants from the matrix of the same representation of $\SU n$.

Our investigation into the special orthogonal group $\SO 4$ produces some new examples of proper biharmonic functions on open subsets of the $3$-dimensional round sphere $\s ^3$.  The existence of proper biharmonic functions on $3$-dimensional space forms was already studied in \cite{Caddeo}, where the radially symmetric examples were constructed. In Theorem \ref{theo-duality} we develop a general duality principle which we use to construct new proper biharmonic functions on the non-compact $3$-dimensional hyperbolic space $\H ^3$.

We conclude the paper with a short appendix providing a complementary approach that hopefully will be useful to implement our calculations by means of some suitable software.

\section{Proper $r$-harmonic functions}\label{section-r-harmonic}

Let $(M,g)$ be a smooth manifold equipped with a semi-Riemannian metric $g$.  We complexify the tangent bundle $TM$ of $M$ to $T^{\cn}M$ and extend the metric $g$ to a complex-bilinear form on $T^{\cn}M$.  Then the gradient $\nabla f$ of a complex-valued function $f:(M,g)\to\cn$ is a section of $T^{\cn}M$.  In this situation, the well-known linear {\it Laplace-Beltrami} operator (alt. tension field) $\tau$ on $(M,g)$ acts on $f$ as follows
$$
\tau(f)=\Div (\nabla f)=\frac{1}{\sqrt{|g|}} \frac{\partial}{\partial x_j}\left(g^{ij}\, \sqrt{|g|}\, \frac{\partial f}{\partial x_i}\right).
$$
For two complex-valued functions $f,h:(M,g)\to\cn$ we have the following well-known relation
$$\tau(fh)=\tau(f)\,h+2\,\kappa(f,h)+f\,\tau(h),$$
where the {\it conformality} operator $\kappa$ is given by
$$
\kappa(f,h)=g(\nabla f,\nabla h).
$$
For a positive integer $r$, the iterated Laplace-Beltrami operator $\tau^r$ is defined by
$$
\tau^{0} (f)=f,\quad \tau^r (f)=\tau(\tau^{(r-1)}(f)).
$$
\begin{definition}\label{definition-proper-r-harmonic} For a positive integer $r$, we say that a complex-valued function $f:(M,g)\to\cn$ is
\begin{enumerate}
\item[(a)] {\it $r$-harmonic} if $\tau^r (f)=0$,
\item[(b)] {\it proper $r$-harmonic} if $\tau^r (f)=0$ and $\tau^{(r-1)} (f)$ does not vanish identically.
\end{enumerate}
\end{definition}

It should be noted that the {\it harmonic} functions are exactly the $1$-harmonic
and the {\it biharmonic} functions are the $2$-harmonic ones.  In some texts, the $r$-harmonic functions are also called {\it polyharmonic} of order $r$.

We shall now develop some connections between the theory of $r$-harmonic functions and the notion of harmonic morphisms.  More specifically, we recall that a map $\pi:(\hat M,\hat g)\to(M,g)$ between two semi-Riemannian manifolds is a \emph{harmonic morphism} if it pulls back germs of harmonic functions to germs of harmonic functions. The standard reference on this topic is the book \cite{BW-book} of Baird and Wood.   We also recommend the updated online bibliography \cite{Gud-bib}.

\begin{proposition}\label{prop:lift-tension}
Let $\pi:(\hat M,\hat g)\to (M,g)$ be a submersive harmonic morphism
from a semi-Riemannian manifold $(\hat M,\hat g)$ to a Riemannian manifold
$(M,g)$. Further let $f:(M,g)\to\C$ be a smooth function and
$\hat f:(\hat M,\hat g)\to\C$ be the composition $\hat f=f\circ\pi$.
If $\lambda:\hat M\to\rn^+$ is the dilation of $\pi$ then the tension
field satisfies
$$
\tau(f)\circ\pi=\lambda^{-2}\tau(\hat f)\ \ \text{and}\ \ \tau^r(f)\circ\pi=\lambda^{-2}\tau(\lambda^{-2}\tau^{(r-1)}(\hat f))
$$
for all positive integers $r\ge 2$.
\end{proposition}

\begin{proof}
The harmonic morphism $\pi$ is a horizontally conformal, harmonic map, see \cite{Fuglede96} or \cite{BW-book}. Hence the well-known composition law for the tension field gives
\begin{eqnarray*}
\tau(\hat f)
&=&\tau(f\circ\pi)\\
&=&\text{trace}\nabla df(d\pi,d\pi)+ df(\tau(\pi))\\
&=&\lambda^2\tau(f)\circ\pi+df(\tau(\pi))\\
&=&\lambda^2\tau(f)\circ\pi.
\end{eqnarray*}
For the second statement, set $h=\tau(f)$ and
$\hat h=\lambda^{-2}\cdot\tau(\hat f)$.
Then $\hat h=h\circ\pi$ and it follows from the first step that
$$\tau(\lambda^{-2}\tau(\hat f))=\tau(\hat h)
=\lambda^2\tau(h)\circ\pi=\lambda^2\tau^2(f)\circ\pi,$$
or equivalently,
$$\tau^2(f)\circ\pi=\lambda^{-2}\tau(\lambda^{-2}\tau(\hat f)).$$
The rest follows by induction.
\end{proof}

\begin{example}\label{example-proper-biharmonic-S^3}
We equip $\R^4$ with its standard Euclidean metric, satisfying
$$
(x,y)=x_1y_1+x_2y_2+x_3y_3+x_4y_4.
$$
The round $3$-dimensional unit sphere $\s^3$ in $\rn^4$ is given by
$$
\s^3=\{(x_1,x_2,x_3,x_4)\in\R^4|\ x_1^2+x_2^2+x_3^2+x_4^2=1\}.
$$
The radial projection $\pi:\R^4\setminus\{0\}\to \s^3$ with $\pi:x\mapsto x/|x|$ is a well-known harmonic morphism and its dilation satisfies $\lambda^{-2}(x)=|x|^2$. Let $p,q\in\cn^4$ be linearly independent, $(q,q)=0$, $f:W\to\C$ be the function defined locally on $\s^3$ with
$$
f(x)=\frac{p_1x_1+\dots +p_4x_4}{q_1x_1+\dots +q_4x_4}
$$
and $\hat f=f\circ\pi$.  Then an easy calculation shows that
$$
\tau(f)=|x|^2\Delta\hat f=-\frac{2|x|^2(p,q)}{(q_1x_1+\dots +q_4x_4)^2}
$$
and
$$\tau^2(f)=|x|^2\Delta(|x|^2\Delta(\hat f))=0.
$$
Here $\Delta$ is the tension field on $\R^4$ i.e. the classical Laplace operator given by
$$
\Delta = \frac{\partial^2}{\partial x_1^2}+\frac{\partial^2}{\partial x_2^2}+\frac{\partial^2}{\partial x_3^2}+\frac{\partial^2}{\partial x_4^2}.
$$
These calculations show that if $(p,q)\neq 0$ then the local function $f:W\to\C$ is proper biharmonic on $\s^3$.
\end{example}

\begin{example}\label{example-hyperbolic}
Let $\R^4_1$ be the standard $4$-dimensional Minkowski space equipped with its Lorentzian metric
$$
(x,y)_L=-x_0y_0+x_1y_1+x_2y_2+x_3y_3.
$$
Bounded by the light cone, the open set
$$
U=\{x\in\R^4_1|\ (x,x)_L<0\ \text{and}\ 0<x_0\}
$$
contains the 3-dimensional hyperbolic space
$$
\H^3=\{(x_0,x_1,x_2,x_3)\in\R^4_1|\ (x,x)_L=-1\ \text{and}\ 0<x_0\}.
$$
Let $\pi:U\to \H^3$ be the radial projection given by
$$
\pi:x\mapsto \frac{x}{\sqrt{-(x,x)_L}}.
$$
This is a harmonic morphism and its dilation satisfies $\lambda^{-2}(x)=-|x|^2_L$, see \cite{Gud-7}. Let $p,q\in\cn^4_1$ be linearly independent, $(q,q)_L=0$, $f:W\to\C$ be the function defined locally on $\H^3$ with
$$
f(x)=\frac{p_0x_0+\dots +p_3x_3}{q_0x_0+\dots +q_3x_3}
$$
and $\hat f=f\circ\pi$.  Then
$$
\tau(f)=-|x|_L^2\Box\hat f=\frac{2|x|_L^2(p,q)_L}{(q_0x_0+\dots +q_3x_3)^2}$$
and
$$
\tau^2(f)=-|x|_L^2\Box(-|x|_L^2\Box(\hat f))=0.
$$
Here $\Box$ is the tension field on $\R^4_1$ i.e. the wave operator of d'Alembert given by
$$
\Box = -\frac{\partial^2}{\partial x_0^2}+\frac{\partial^2}{\partial x_1^2}+\frac{\partial^2}{\partial x_2^2}+\frac{\partial^2}{\partial x_3^2}.
$$
From this we see that if $(p,q)_L\neq 0$ then $f$ is proper biharmonic. It should be noted that $q\in\cn^4$ can easily be chosen such that $f:\H^3\to\cn$ is globally defined.
\end{example}

In the sequel, we shall often employ the following immediate
consequence of Proposition~\ref{prop:lift-tension}.

\begin{corollary}\label{prop:lift-proper-k-harmonic}
Let $\pi:(\hat M,\hat g)\to (M,g)$ be a submersive harmonic morphism,
from a semi-Riemannian manifold $(\hat M,\hat g)$ to a Riemannian manifold
$(M,g)$, with constant dilation.  Further let $f:(M,g)\to\C$ be a smooth function and
$\hat f:(\hat M,\hat g)\to\C$ be the composition $\hat f=f\circ\pi$.
Then the following statements are equivalent
\begin{enumerate}
\item[(a)] $\hat f:(\hat M,\hat g)\to\C$ is proper $r$-harmonic,
\item[(b)] $f:(M,g)\to\C$ is proper $r$-harmonic.
\end{enumerate}
\end{corollary}

\begin{proof}
It follows from Proposition~\ref{prop:lift-tension} that for any positive integer $r$ we have
$$
\tau^r(f)\circ\pi=\lambda^{-2r}\tau^r(\hat f).
$$
The statement is a direct consequence of these relations.
\end{proof}

\begin{remark}
The special case $r=2$ in Corollary~\ref{prop:lift-proper-k-harmonic} is partially a consequence of Theorem 3.1 in \cite{Ou96} (see also \cite{LoubeauOu96}).
\end{remark}

\begin{example}\label{ex-riemannian-submersions-biharmonic}
Let $G$ be a Lie group,  $K\subset H$ be compact subgroups of $G$ and $\k,\h,\g$ be their Lie algebras, respectively.  Then we have
the homogeneous fibration $$\pi:G/K\to G/H,\ \ \ \pi:aK\mapsto aH,$$ with fibres diffeomorphic to $H/K$.  Let $\m$ be an $\text{Ad}(H)$-invariant complement of $\h$ in $\g$ and $\p$ be an $\text{Ad}(K)$-invariant complement of $\k$ in $\h$. Then $\p\oplus\m$ is an $\text{Ad}(K)$-invariant complement of $\k$ in $\g$.

Let $<,>_{\m}$ be an $\text{Ad}(H)$-invariant scalar product on $\m$ inducing a $G$-invariant Riemannian metric $g$ on the homogeneous space $G/H$. Further let $<,>_{\p}$ be an $\text{Ad}(K)$-invariant scalar product on $\p$ defining a $H$-invariant Riemannian metric $\bar g$ on $H/K$. Then the orthogonal sum $$<,>=<,>_{\m}+<,>_{\p}$$ on $\m\oplus\p$ defines a $G$-invariant Riemannian metric $\hat g$ on $G/K$. It is well-known that the homogeneous projection $\pi:(G/K,\hat g)\to (G/H,g)$ is a Riemannian submersion with totally geodesic fibres see \cite{B-B-1} or \cite{Bes}. This implies that $\pi$ is a harmonic morphisms with constant dilation $\lambda\equiv 1$.

For this general situation, we have the following important examples from the special unitary, the special orthogonal and the quaternionic unitary groups.  Here the groups are equipped with their standard biinvariant Riemannian metrics induced by their Killing forms.
$$
\SU {n_1+\cdots +n_k}\to \SU{n_1+\cdots +n_k}/
\text{\bf S}(\U{n_1}\times\cdots\times\U{n_k}),
$$
$$
\SO {n_1+\cdots +n_k}\to \SO {n_1+\cdots +n_k}/
\SO{n_1}\times\cdots\times\SO{n_k},
$$
$$
\Sp {n_1+\cdots +n_k}\to \Sp {n_1+\cdots +n_k}/
\Sp{n_1}\times\cdots\times\Sp{n_k}.
$$

By applying Corollary~\ref{prop:lift-proper-k-harmonic} to the case $\SO {4}\to \SO {4}/ \SO{3}=\s^3$, we deduce that the proper biharmonic functions on $\s ^3$, described in Example~\ref{example-proper-biharmonic-S^3}, lift to proper biharmonic functions on $\SO{4}$.
\end{example}

\section{The Riemannian Lie group $\GLC n$}

Let $G$ be a Lie group with Lie algebra $\g$ of left-invariant
vector fields on $G$.  Then a Euclidean scalar product $g$ on $\g$
induces a left-invariant Riemannian metric on the
group $G$ and turns it into a homogeneous Riemannian manifold. If $Z$ is a
left-invariant vector field on $G$ and $f,h:U\to\cn$ are two
complex-valued functions defined locally on $G$ then the first and
second order derivatives satisfy
\begin{equation}\label{eq:derivativeZ}
Z(f)(p)=\frac {d}{ds}[f(p\cdot\exp(sZ))]\big|_{s=0},
\end{equation}
\begin{equation}\label{eq:derivativeZZ}
Z^2(f)(p)=\frac {d^2}{ds^2}[f(p\cdot\exp(sZ))]\big|_{s=0}.
\end{equation}

Further, assume that $G$ is a subgroup of the complex general linear
group $\GLC n$ equipped with its standard Riemannian metric.  This
is induced by the Euclidean scalar product on the Lie algebra $\glc n$ given by
$$
g(Z,W)=\Re\trace ZW^*.
$$
Employing the Koszul formula for the Levi-Civita connection $\nabla$
on $\GLC n$, we see that
\begin{eqnarray*}
g(\nab ZZ,W)&=&g([W,Z],Z)\\
&=&\Re\trace (WZ-ZW)Z^t\\
&=&\Re\trace W(ZZ^t-Z^tZ)^t\\
&=&g([Z,Z^t],W).
\end{eqnarray*}
Let $[Z,Z^t]_\g$ be the orthogonal projection of the bracket
$[Z,Z^t]$ onto the subalgebra $\g$ of $\glc n$.  Then the above calculations shows
that $$\nab ZZ=[Z,Z^t]_\g.$$
This implies that the tension field $\tau(f)$ and the conformality operator
$\kappa(f,h)$ are given by
\begin{equation}\label{tau-kappa-alie-groups}
\tau(f)=\sum_{Z\in\B}Z^2(f)-[Z,Z^t]_\g(f)
\ \ \text{and}\ \
\kappa(f,h)=\sum_{Z\in\B}Z(f)Z(h),
\end{equation}
where $\B$ is any orthonormal basis for the Lie algebra $\g$.

\begin{remark}
For $1\le i,j\le n$ we shall denote by
$E_{ij}$  the element of $\glr n$ satisfying
$$(E_{ij})_{kl}=\delta_{ik}\delta_{jl}$$ and by $D_t$ the diagonal
matrices $$D_t=E_{tt}.$$ For $1\le r<s\le n$ let $X_{rs}$ and
$Y_{rs}$ be the matrices satisfying
$$X_{rs}=\frac 1{\sqrt 2}(E_{rs}+E_{sr}),\ \ Y_{rs}=\frac
1{\sqrt 2}(E_{rs}-E_{sr}).$$
\end{remark}

\section{The special unitary group $\SU n$}\label{section-SU(n)}

In this section we construct proper biharmonic functions on open subsets of the special unitary group $\SU n$.  They are quotients of first order homogeneous polynomials in the matrix coefficients of the standard $n$-dimensional representation $\pi_1$ of $\SU n$. The unitary group $\U n$ is the compact subgroup of $\GLC n$ given by
$$
\U n=\{z\in\GLC{n}|\ z\cdot z^*=I_n\},
$$
with its standard matrix representation
$$
z=\begin{bmatrix}
z_{11} & z_{12} & \cdots & z_{1n}\\
z_{21} & z_{22} & \cdots & z_{2n}\\
\vdots & \vdots & \ddots & \vdots \\
z_{n1} & z_{n1} & \cdots & z_{nn}
\end{bmatrix}.
$$
The circle group $\s^1=\{e^{i\theta}\in\cn | \ \theta\in\rn\}$ acts on the unitary group $\U n$ by multiplication
$$
(e^{i\theta},z)\mapsto e^{i\theta} z
$$
and the orbit space of this action is the special unitary group
$$
\SU n=\{ z\in\U n|\ \det z = 1\}.
$$
The natural projection $\pi:\U n\to\SU n$ is a harmonic morphism with contant dilation $\lambda\equiv 1$.

The Lie algebra $\u n$ of the unitary group $\U n$ satisfies
$$
\u{n}=\{Z\in\cn^{n\times n}|\ Z+Z^*=0\}
$$
and for this we have the canonical orthonormal basis
$$
\{Y_{rs}, iX_{rs}|\ 1\le r<s\le n\}\cup\{iD_t|\ t=1,\dots ,n\}.
$$
Now, by means of a direct computation based on \eqref{eq:derivativeZ}, \eqref{eq:derivativeZZ} and \eqref{tau-kappa-alie-groups}, we have the following basic result, see \cite{Gud-Sak-1}.

Note that from now on we shall use latin indices for rows and greek indices for columns.

\begin{lemma}\label{lemm:complex}
For $1\le j,\alpha\le n$, let $z_{j\alpha}:\U n\to\cn$ be the complex-valued matrix coefficients of the standard representation of $\U n$ given by
$$
z_{j\alpha}:z\mapsto e_j\cdot z\cdot e_\alpha^t,
$$
where $\{e_1,\dots ,e_n\}$ is the canonical basis for $\cn^n$. Then the following relations hold
\begin{equation}\label{lemma5-1}
\tau(z_{j\alpha})= -n\cdot z_{j\alpha}\ \ \text{and}\ \ \kappa(z_{j\alpha},z_{k\beta})= -z_{k\alpha}z_{j\beta}.
\end{equation}
\end{lemma}

We can now state our first construction of complex-valued biharmonic functions.

\begin{theorem}\label{teo-quotient-linear}
Let $p,q\in\cn^n$ be linearly independent and $P,Q:\U n\to\cn$ be the complex-valued functions on the unitary group given by
$$
P(z)=\sum_{j}p_jz_{j\alpha}\ \ \text{and}\ \ Q(z)=\sum_{k}q_kz_{k\beta}.
$$
Further, let the rational function $f(z)=P(z)/Q(z)$ be defined on the open and dense subset $W_Q=\{z\in\U n|\ Q(z)\neq 0\}$ of the unitary group.  Then the following is true.
\begin{itemize}
\item[(a)] The function $f$ is harmonic if and only if $\alpha=\beta$.
\item[(b)] The function $f$ is proper biharmonic if and only if $\alpha\neq\beta$.
\end{itemize}
The corresponding statements hold for the function induced on $\SU n$.
\end{theorem}

\begin{proof}
It is easily seen that for a general quotient $f=P/Q$ we have
\begin{equation}\label{tau-quotient}
Q^3\tau(f)=Q^2\tau(P)-2Q\kappa(P,Q)+2P\kappa(Q,Q)-PQ\tau(Q).
\end{equation}
It follows from Lemma~\ref{lemm:complex} that $P,Q:\U n\to \cn$ are eigenfunctions of the Laplace-Beltrami operator $\tau$ and that $\kappa(Q,Q)=-Q^2$.  This implies
\begin{equation}\label{eq:tau-f-un}
\tau(f)=-2f-2\kappa(P,Q)Q^{-2}.
\end{equation}
A simple calculation shows that
$$
\kappa(P,Q)=-\sum_{j, k} p_j\, q_k\, z_{j\beta}\, z_{k\alpha},
$$
and equation \eqref{eq:tau-f-un} tells us that $f$ is harmonic if and only if $\kappa(P,Q)=-P Q$.
Since $p,q\in\cn^n$ are linearly independent, this holds if and only if $\alpha=\beta$.

If we now assume that $\alpha\neq\beta$ then, again using Lemma~\ref{lemm:complex}, we yield
\begin{eqnarray}\label{varie-th-U(n)}\nonumber
\tau(\kappa(P,Q))&=&2PQ-2n\kappa(P,Q), \\
\kappa(\kappa(P,Q),Q^{-2})&=&4\kappa(P,Q)Q^{-2}, \\ \nonumber
\tau(Q^{-2})&=&2(n-3)Q^{-2}.\\ \nonumber
\end{eqnarray}

We prove statement (b) by computing the bitension field $\tau^2(f)$ using \eqref{varie-th-U(n)}.
\begin{eqnarray*}
\tau^2(f)&=&-2\tau(f)-2\tau(\kappa(P,Q)Q^{-2})\\
&=&4f+4\kappa(P,Q)Q^{-2}-2\tau(\kappa(P,Q)Q^{-2})\\
&=&4f+4\kappa(P,Q)Q^{-2}-2\tau(\kappa(P,Q))Q^{-2}\\
& &\qquad \qquad -4\kappa(\kappa(P,Q),Q^{-2})-2\kappa(P,Q)\tau(Q^{-2})\\
&=&4f-4f+(4+4n-4n-16+12)\kappa(P,Q)Q^{-2}\\
&=&0.
\end{eqnarray*}
The last statement of the theorem is a simple consequence of the fact that the function $f$ is invariant under the action of $\s^1$ on $\U n$.
\end{proof}

\begin{example}\label{example-from-S^3}
As a special case of Theorem~\ref{teo-quotient-linear}, let us assume that $n=2$ and consider the proper biharmonic function $f$ defined locally on $\U 2$ by
\begin{equation}\label{example}\nonumber
f(z)= \frac{z_{11}}{z_{22}}.
\end{equation}
This function is invariant under multiplication by $e^{i \theta}$ on ${\U 2}$ so it induces a proper biharmonic function defined locally on ${\SU 2}$.  It is well-known that $\SU 2$ is, up to a constant multiple of the metric, isometric to $\s^3$ as the Lie group of unit quaternions via
$$
(z,w)\in\s^3\mapsto
\begin{bmatrix}z & w\\ -\bar w & \bar z\end{bmatrix}
=\begin{bmatrix}x_1+ix_2 & x_3+ix_4\\ -x_3+ix_4 & x_1-ix_2\end{bmatrix}\in\SU 2.
$$
If we write the function $f$ in terms of $(z,w)=(x_1+ix_2,x_3+ix_4)$ we get
$$
f(x_1,x_2,x_3,x_5)=f(z,w)=\frac z{\bar z}=\frac{x_1+ix_2}{x_1-ix_2}
=\frac{x_1^2-x_2^2}{x_1^2+x_2^2}+i\frac{2x_1x_2}{x_1^2+x_2^2}.$$
It should be noted that $f$ is exactly the function which we obtained in Example \ref{example-proper-biharmonic-S^3} by choosing $p=(1,i,0,0)$ and $q=(1,-i,0,0)$.
\end{example}

\section{The special orthogonal group $\SO n$}\label{section-SO(n)}

In this section we construct proper biharmonic functions on open subsets of the special orthogonal group $\SO n$. They are quotients of first order homogeneous polynomials in the matrix coefficients of its standard representation. The Lie group $\SO n$ is the subgroup of $\GLR n$ given by
$$
\SO{n}=\{x\in\GLR{n}\ |\ x\cdot x^t=I_n,\ \det x =1\}.
$$
Its Lie algebra $\so n$ is the set of skew-symmetric matrices
$$
\so n=\{X\in\glr n |\ X+X^t=0\}
$$
and for this we have the canonical orthonormal basis
$$
\{Y_{rs}|\ 1\le r<s\le n\}.
$$

For the special orthogonal group we have the following basic result, see \cite{Gud-Sak-1}.

\begin{lemma}\label{lemm:real}
For $1\le j,\alpha\le n$, let $x_{j\alpha}:\SO n\to\rn$ be the real-valued matrix coefficients of the standard representation of $\SO n$ given by
$$
x_{j\alpha}:x\mapsto e_j\cdot x\cdot e_\alpha^t,
$$
where $\{e_1,\dots ,e_n\}$ is the canonical basis for $\rn^n$. Then the following relations hold
$$
\tau(x_{j\alpha})=-\frac {(n-1)}2\cdot x_{j\alpha},
$$
$$
\kappa(x_{j\alpha},x_{k\beta})=-\frac 12\cdot (x_{k\alpha}x_{j\beta}-\delta_{kj}\delta_{\alpha\beta}).
$$
\end{lemma}

The next result describes our construction of complex-valued biharmonic functions from open subsets of $\SO n$.

\begin{theorem}\label{theorem-orthogonal-group}
Let $p,q\in\cn^n$ be linearly independent and $P,Q:\SO n\to\cn$ be the complex-valued functions on the special orthogonal group given by
$$
P(x)=\sum_{j}p_jx_{j\alpha},\ \ Q(x)=\sum_{k}q_kx_{k\beta}.
$$
Further, let the rational function $f(x)=P(x)/Q(x)$ be defined on the open and dense subset $W_Q=\{x\in\SO n|\ Q(x)\neq 0\}$ of the special orthogonal group.  The function $f$ is harmonic if and only if $\alpha=\beta$, $(q,q)=0$ and $(p,q)=0$. It is proper biharmonic if and only if
\begin{enumerate}
\item[(a)] $\alpha=\beta$, $(q,q)=0$, $(p,q)\neq 0$ and $n=4$, or
\item[(b)] $\alpha\neq\beta$, $(q,q)=0$ and $(p,q)=0$.
\end{enumerate}
\end{theorem}

\begin{proof}
First we note that $P,Q:\SO n\to \cn$ are eigenfunctions of the Laplace-Beltrami
operator $\tau$ and so the general relation \eqref{tau-quotient} simplifies to
\begin{equation}\label{simplified-general-equation}\nonumber
Q^3\tau(f)=2P\kappa(Q,Q)-2Q\kappa(P,Q).
\end{equation}
We first consider the case when $\alpha=\beta$.  Then Lemma~\ref{lemm:real} gives
$$
\kappa(Q,Q)=\frac 12((q,q)-Q^2)\ \ \text{and}\ \ \kappa(P,Q)=\frac 12((p,q)-PQ),
$$
hence
$$
\tau(f)=(q,q)PQ^{-3}-(p,q)Q^{-2}.
$$
This implies that if $\alpha=\beta$ then $f$ is harmonic if and only if $(q,q)=0$ and $(p,q)=0$.  A simple calculation shows that
\begin{eqnarray}\label{eq:tauq-2q-3}\nonumber
\tau(Q^{-2})&=&(n-4)Q^{-2}+3(q,q)Q^{-4},\\
\tau(Q^{-3})&=&\frac 32(n-5)Q^{-3}+6(q,q)Q^{-5}\\ \nonumber
\kappa(P,Q^{-3})&=&\frac 32PQ^{-3}-\frac 32(p,q)Q^{-4}.
\end{eqnarray}
This leads us to
\begin{eqnarray}\label{final-step-case1}
\tau^2(f)&=&(q,q)\tau(PQ^{-3})-(p,q)\tau(Q^{-2})\\ \nonumber
&=&(q,q)(\tau(P)Q^{-3}+2\kappa(P,Q^{-3})+P\tau(Q^{-3}))-(p,q)\tau(Q^{-2})\\ \nonumber
&=&-6(p,q)(q,q)Q^{-4}-(p,q)(n-4)Q^{-2}\\ \nonumber
& &\qquad\qquad +6(q,q)^2PQ^{-5}+(q,q)(n-4)PQ^{-3}.
\end{eqnarray}
An inspection of \eqref{final-step-case1}, using the fact that the right-hand side is the sum of terms which are not homogeneous, enables us to conclude that $\tau^2(f)=0$ and $\tau(f)\neq 0$ if and only if $(q,q)=0$, $(p,q)\neq 0$ and $n=4$.

Let us now assume that $\alpha\neq\beta$.  Then we apply Lemma~\ref{lemm:real} and obtain the following
\begin{eqnarray}\label{eq1:alphanotbeta}\nonumber
\kappa(Q,Q)&=&\frac 12((q,q)-Q^2), \\ \nonumber
\kappa(P,Q)&=&-\frac 12(\sum_{j}p_jx_{j\beta})(\sum_{k}q_kx_{k\alpha}),\\
\tau(\kappa(P,Q))&=&\frac 12PQ-(n-1)\kappa(P,Q),\\ \nonumber
\kappa(\kappa(P,Q),Q^{-2})&=&2\kappa(P,Q)Q^{-2}+(p,q) \, \frac{1}{2}\, (\sum_{k}q_k x_{k\alpha})\, Q^{-3},\\ \nonumber
\tau(Q^{-2})&=&(n-4)Q^{-2}. \nonumber
\end{eqnarray}
The tension field
$$
\tau(f)=(q,q)PQ^{-3}-f-2\kappa(P,Q)Q^{-2}
$$
does not vanish identically since since $p$ and $q$ are linearly independent, so $f$ is not harmonic in this case.  Now we use \eqref{eq1:alphanotbeta} and compute
\begin{eqnarray*}
& &\tau(f)+2\tau(\kappa(P,Q)Q^{-2})\\
&=&(q,q) P Q^{-3}-f-2\kappa(P,Q)Q^{-2}+2\tau(\kappa(P,Q))Q^{-2}\\
&&+2\kappa(P,Q)\tau(Q^{-2})+4\kappa(\kappa(P,Q),Q^{-2})\\
&=&(q,q) P Q^{-3} -\frac{P}{Q} -2\kappa(P,Q)Q^{-2} +2\left[\frac{1}{2} PQ- (n-1)\kappa(P,Q)\right]Q^{-2}\\
&&+2 (n-4) \kappa(P,Q) Q^{-2}+8 \kappa(P,Q) Q^{-2} +2(p,q)\, (\sum_{k}q_k x_{k\alpha})\, Q^{-3}\\
&=&(q,q) P Q^{-3} +2(p,q)\, (\sum_{k}q_k x_{k\alpha})\, Q^{-3}
\end{eqnarray*}
Taking \eqref{eq:tauq-2q-3} and \eqref{eq1:alphanotbeta} into account, it follows that
\begin{eqnarray*}\label{bitension-orthogonal-group-last}
\tau^2(f)&=&(q,q)\tau(PQ^{-3})-(q,q) P Q^{-3}-2(p,q)\, (\sum_{k}q_k x_{k\alpha})\, Q^{-3}\\
&=&(q,q)[\tau(P)Q^{-3}+2\kappa(P,Q^{-3})+P\tau(Q^{-3})]\\
& &\qquad\qquad\quad\qquad -(q,q) P Q^{-3}-2(p,q)\, (\sum_{k}q_k x_{k\alpha})\, Q^{-3}\nonumber\\
&=&(q,q)[(n-8)PQ^{-3}-6\kappa(P,Q)Q^{-4}+6(q,q)PQ^{-5}]\\
& & \qquad\qquad\qquad\qquad\qquad -2(p,q)\, (\sum_{k}q_k x_{k\alpha})\, Q^{-3}\nonumber\\
&=&\frac{1}{Q^5}\{(q,q)[(n-8)PQ^2-6\kappa(P,Q)Q+6(q,q)P]\\
& &\qquad\qquad\qquad\qquad -2(p,q)\, (\sum_{k}q_k x_{k\alpha})\, Q^{2}\}\nonumber
\end{eqnarray*}
Now it is obvious that if $(q,q)=0$ and $(p,q)=0$ then the bitension field vanishes. Since the only polynomial of degree one in the numerator of the last equation is $6(q,q)^2P$, it is clear that the vanishing of the bitension field implies $(q,q)=0$. From this it is immediate to deduce that also $(p,q)=0$, so the converse is also true.
\end{proof}

\begin{remark}
We point out that the solutions provided by Theorem~\ref{theorem-orthogonal-group} in the case that $\alpha=\beta$ and $n=4$ are precisely those given at the end of Example~\ref{ex-riemannian-submersions-biharmonic}.
\end{remark}

\section{The quaternionic unitary group $\Sp n$}\label{section-Sp(n)}

In this section we construct complex-valued proper biharmonic functions on open and dense subsets of the quaternionic unitary group $\Sp n$ i.e. the intersection of the unitary group $\U{2n}$ and the standard representation of the quaternionic general linear group $\GLH n$ in $\cn^{2n\times 2n}$ given by
$$
(z+jw)\mapsto q=\begin{bmatrix}z & w \\ -\bar w & \bar
z\end{bmatrix}.
$$
The Lie algebra $\sp n$ of $\Sp n$ satisfies
$$
\sp{n}=\{\begin{bmatrix} Z & W
\\ -\bar W & \bar Z\end{bmatrix}\in\cn^{2n\times 2n}
\ |\ Z^*+Z=0,\ W^t-W=0\}
$$
and for this we have the standard orthonormal basis which is the union of the following three sets
$$
\{\frac 1{\sqrt 2}\begin{bmatrix}Y_{rs} & 0 \\
0 & Y_{rs}\end{bmatrix},\frac 1{\sqrt 2}\begin{bmatrix}iX_{rs} & 0 \\
0 & -iX_{rs}\end{bmatrix}\ |\ 1\le r<s\le n\},
$$
$$
\{\frac 1{\sqrt 2}\begin{bmatrix}0 & iX_{rs} \\
iX_{rs} & 0\end{bmatrix},\frac 1{\sqrt 2}\begin{bmatrix}0 & X_{rs} \\
-X_{rs} & 0\end{bmatrix}\ |\ 1\le r<s\le n\},
$$
$$
\{\frac 1{\sqrt 2}\begin{bmatrix}iD_{t} & 0 \\
0 & -iD_{t}\end{bmatrix},\frac 1{\sqrt 2}\begin{bmatrix}0 & iD_{t}  \\
iD_{t} & 0\end{bmatrix},\frac 1{\sqrt 2}\begin{bmatrix}0 & D_{t}  \\
-D_{t} & 0\end{bmatrix}\ |\ 1\le t\le n\}.
$$

For the quaternionic unitary group $\Sp n$ we have the following basic result,
which is an improvement of Lemma 6.1 of \cite{Gud-Sak-1}.

\begin{lemma}\label{lemm:quaternionic}
For $1\le j,k,\alpha,\beta \le n$, let $z_{j\alpha},w_{k \beta}:\Sp n\to\cn$ be the complex valued matrix coefficients of the standard representation of $\Sp n$ given by
$$
z_{j\alpha}:q\mapsto e_j\cdot q\cdot e_\alpha^t,\ \
w_{k \beta}:q\mapsto e_{k}\cdot q\cdot e_{n+\beta}^t,
$$
where $\{e_1,\dots ,e_{2n}\}$ is the canonical basis for $\cn^{2n}$. Then the following relations hold
$$
\tau(z_{j\alpha})= -\frac{2n+1}2\cdot z_{j\alpha},\ \ \tau(w_{k \beta})=
-\frac{2n+1}2\cdot w_{k \beta},
$$
$$
\kappa(z_{j\alpha},z_{k\beta})=-\frac 12\cdot z_{k\alpha}z_{j\beta},\ \
\kappa(w_{j\alpha},w_{k\beta})=-\frac 12\cdot w_{k\alpha}w_{j\beta},
$$
$$
\kappa(z_{j\alpha},w_{k\beta})=-\frac 12\cdot z_{k\alpha}w_{j\beta}.
$$
\end{lemma}

\begin{proof}
The first four relations were proven in \cite{Gud-Sak-1}.  Since the quaternionic unitary group $\Sp n$ is a subgroup of $\U {2n}$, a generic element
$$
\begin{bmatrix}z & w \\ -\bar w & \bar z\end{bmatrix}\in\Sp n
$$
satisfies
$$
\begin{bmatrix}z & w \\ -\bar w & \bar z\end{bmatrix}
\begin{bmatrix}z^* & -w^t \\ w* & z^t\end{bmatrix}=
\begin{bmatrix}zz^*+ww^* & wz^t-zw^t \\ \bar zw^*-\bar wz^* & \bar zz^t+\bar ww^t\end{bmatrix}=
\begin{bmatrix}I_n & 0 \\ 0 & I_n\end{bmatrix}.
$$
The equation $wz^t-zw^t=0$ shows that the formula
$$
\kappa(z_{j\alpha},w_{k\beta})=-\frac 12\cdot\big[z_{k\alpha}w_{j\beta}
-{\delta_{\alpha\beta}}\cdot \sum_{\tau=1}^n(z_{j\tau}w_{k\tau}
-z_{k\tau}w_{j\tau})\big],
$$
from Lemma 6.1 of \cite{Gud-Sak-1}, simplifies to
$$
\kappa(z_{j\alpha},w_{k\beta})=-\frac 12\cdot z_{k\alpha}w_{j\beta}.
$$
\end{proof}

In the spirit of the previous sections, we now look for proper biharmonic functions of the type $f=P\slash Q$, where $P$ and $Q$ are suitable homogeneous polynomials.

\begin{theorem}\label{teo-quotient-linear-quaternionic}
Let $p,q\in\cn^{2n}$ be linearly independent and $P,Q:\Sp n\to\cn$ be the complex-valued functions on the quaternionic unitary group defined by
$$
P(z,w)=\sum_{j}(p_j z_{j\alpha}+p_{n+j} w_{j\alpha})\ \ \text{and}\ \
Q(z,w)=\sum_{k}(q_k z_{k\beta}+q_{n+k}w_{k\beta}).
$$
Further, let the rational function $f(z,w)=P(z,w)/Q(z,w)$ be defined on the open and dense subset $W_Q=\{z+jw \in\Sp n|\ Q(z,w)\neq 0\}$ of the quaternionic unitary group.
\begin{itemize}
\item[(a)] The function $f$ is harmonic if and only if $\alpha=\beta$.
\item[(b)] The function $f$ is proper biharmonic if and only if $\alpha\neq\beta$.
\end{itemize}
\end{theorem}

\begin{proof}
It follows from Lemma~\ref{lemm:quaternionic} that $P,Q:\U n\to \cn$ are eigenfunctions of the Laplace-Beltrami operator $\tau$ and that $\kappa(Q,Q)=-Q^2/2$.  Then we deduce from the general formula \eqref{tau-quotient} that
\begin{equation}\label{eq:tau-f-spn}
\tau(f)=-f-2\kappa(P,Q)Q^{-2}.
\end{equation}
Again applying Lemma~\ref{lemm:quaternionic}, we obtain
\begin{equation}\label{eq:kpq-spn}
\kappa(P,Q)=-\frac{1}{2}\, \sum_{j,k}(p_jz_{j\beta}+p_{n+j}w_{j\beta}) (q_k z_{k\alpha}+q_{n+k}w_{k\alpha}).
\end{equation}
Equation \eqref{eq:tau-f-spn} tells us that $f$ is harmonic if and only if $\kappa(P,Q)=-P Q/2.$
It now follows from \eqref{eq:kpq-spn} that this is true if and only if $\alpha=\beta$.

For the case $\alpha\neq\beta$, a standard calculation, employing Lemma~\ref{lemm:quaternionic}, yields
\begin{eqnarray}\label{eq:quat-1}\nonumber
\tau (Q^{-2})&=& 2(n-1) Q^{-2}, \\
\tau(\kappa(P,Q))&=&\frac{1}{2}PQ-(2n+1) \kappa(P,Q), \\ \nonumber
\kappa(\kappa(P,Q),Q^{-2})&=&2\kappa(P,Q)Q^{-2}. \\ \nonumber
\end{eqnarray}
Finally, by using \eqref{eq:quat-1}, we obtain the stated result. More precisely,
\begin{eqnarray*}
\tau^2(f)
&=&-\tau(f)-2\tau(\kappa(P,Q)Q^{-2})\\
&=&f+2\kappa(P,Q)Q^{-2}-2\tau(\kappa(P,Q)Q^{-2})\\
&=&f+2\kappa(P,Q)Q^{-2}-2\tau(\kappa(P,Q))Q^{-2}\\
& &\qquad\qquad\qquad -2\kappa(P,Q)\tau(Q^{-2})-4\kappa(\kappa(P,Q),Q^{-2})\\
&=&f+2\kappa(P,Q)Q^{-2}-f+2(2n+1)\kappa(P,Q)Q^{-2}\\
& &\qquad\qquad\qquad -4(n-1)\kappa(P,Q)Q^{-2}-8\kappa(P,Q)Q^{-2}\\
&=& (2+4n+2-4n+4-8)\kappa(P,Q)Q^{-2} \\
&=&0.
\end{eqnarray*}
\end{proof}

\section{The special unitary group $\SU n$ revisited}

In this section we construct proper biharmonic functions on open subsets of the special unitary group $\SU n$.  They are quotients of $2\times2$ determinants from the matrix for the standard representation of $\SU n$.  For this purpose, we first establish the following result of independent interest.
\begin{proposition}\label{proposition-tau-of-a-determinant-in-U(n)}
Let $d_k$ be a $k\times k$ determinant from the generic element of $\U n$.
Then
\begin{equation}\label{tau-d-k}
\tau(d_k)= - k(n-k+1)d_k.
\end{equation}
\end{proposition}
\begin{proof} We shall use the formula which gives the tension field of a product. More precisely, let
\begin{equation}\label{product-function}\nonumber
f=\prod_{i=1}^k \, f_i \,.
\end{equation}
Then
\begin{equation}\label{tau-of-product-function}
\tau(f)=\sum_{j=1}^k \, \left [\tau(f_j)\,\prod_{i \neq j;\, i=1}^k \, f_i \right ] +2\, \sum_{j<\ell;\,j, \ell=1}^k \,\left [ \kappa(f_j,f_\ell)\prod_{i \neq j,\, i \neq \ell;\, i=1}^k \,f_i \,\right].
\end{equation}
Next, by renumbering rows and columns if necessary, we observe that it is not restrictive to assume that both the row and the column indices of $d_k$ range from $1$ to $k$. In particular, we can write
\begin{equation}\label{determinant-explicit}
d_k= \sum_{\sigma \in \S_k} \, (-1)^{s(\sigma)} \, \, z_{1\,\sigma(1)} \cdots z_{k\,\sigma(k)}\, ,
\end{equation}
where $S_k$ is the set of permutations of $\{1, \ldots,k \}$ and $s(\sigma)$ is the sign of $\sigma$. Now we can compute $\tau(d_k)$. We apply \eqref{tau-of-product-function} to \eqref{determinant-explicit}: by using \eqref{lemma5-1} and the linearity we easily obtain
\begin{equation}\label{tau-determinant-eq:1}
\tau(d_k)=-n\,k\, d_k+\sum_{\sigma \in \S_k} (-1)^{s(\sigma)}\hskip -.3cm\sum_{j<\ell;j, \ell=1}^k\left [ \kappa(z_{j\,\sigma(j)},z_{\ell\,\sigma(\ell)})\hskip -.3cm\prod_{i \neq j, i \neq \ell;\, i=1}^k\hskip -.3cm z_{i\sigma(i)}\right ].
\end{equation}
By using \eqref{lemma5-1} into \eqref{tau-determinant-eq:1} we have
\begin{equation}\label{tau-determinant-eq:2}
\tau(d_k)=-n\,k\, d_k\,-\, 2\,\sum_{\sigma \in \S_k}(-1)^{s(\sigma)}\hskip -.3cm \sum_{j<\ell;\,j, \ell=1}^k \left [ z_{j\,\sigma(\ell)} \, z_{\ell\,\sigma(j)}\hskip -.3cm\prod_{i \neq j,\, i \neq \ell;\, i=1}^k\hskip -.3cm z_{i\,\sigma(i)}\right ] .
\end{equation}
Now we commute the order of the sums. Then, observing that \eqref{lemma5-1} has produced a minus sign together with a change of the sign of the corresponding permutation $\sigma$, it is not difficult to deduce that \eqref{tau-determinant-eq:2} becomes
\begin{eqnarray}\label{tau-determinant-eq:3}\nonumber
\tau(d_k)&=&-n\,k\, d_k\,+\, 2\,\sum_{j<\ell;\,j, \ell=1}^k \,  \left [ \sum_{\sigma \in \S_k} \, (-1)^{s(\sigma)} \, \prod_{ i=1}^k \,\, z_{i\,\sigma(i)} \, \right ] \\ \nonumber
&=&-n\,k\, d_k\,+\, 2\,\sum_{j<\ell;\,j, \ell=1}^k \,  d_k  \\ \nonumber
&=&-n\,k\, d_k\,+\, 2\, \frac{k(k-1)}{2} \, d_k \,\, ,
\end{eqnarray}
from which \eqref{tau-d-k} follows immediately.
\end{proof}

We can now state our main result on the quotient of two determinants.

\begin{theorem}\label{theorem-biharmonicity-2x2-determinants}
Let $P,Q:\U n\to\cn$ be the complex-valued functions on the unitary group given by the $2\times 2$ determinants
$$
P(z)=z_{j\alpha}z_{k\beta}-z_{k\alpha}z_{j\beta}\ \ \text{and}\ \
Q(z)=z_{r\gamma}z_{s \delta}-z_{s \gamma}z_{r\delta}.
$$
Further let $f(z)=P(z)/Q(z)$ be the rational function defined on the open and dense subset $W_Q=\{z\in\U n|\ Q(z)\neq 0\}$
of the unitary group. Then
\begin{enumerate}
\item[(a)] $f$ is harmonic if and only if $(j,k)=(r,s)$ or $(\alpha,\beta)=(\gamma,\delta)$.
\item[(b)] $f$ is properly biharmonic if and only if $(j,k)\neq(r,s)$ and $(\alpha,\beta)\neq(\gamma,\delta)$.
\end{enumerate}
The corresponding statements hold for the functions induced on $\SU n$.
\end{theorem}
\begin{proof}

First, we need to compute the tension field $\tau(f)$. According to Proposition~\ref{proposition-tau-of-a-determinant-in-U(n)}, $P,Q:\U n\to \cn$ are eigenfunctions of the Laplace-Beltrami
operator $\tau$ with eigenvalue $\lambda=-2(n-1)$. Moreover, a simple direct computation  shows that
$$\kappa(Q,Q)=-2 \, Q^2.$$
Now, by using \eqref{tau-quotient} we deduce that
\begin{equation}\label{general-tau}
\tau(f)=-4f-2 \, \kappa(P,Q) \, Q^{-2}.
\end{equation}
Therefore, we now need to write down the explicit expression of $\kappa(P,Q)$. To this purpose, we introduce the following notation. We write
\begin{equation}\label{notation1}
d_{j\alpha k\beta }=  z_{j\alpha}z_{k\beta}-z_{j\beta}z_{k\alpha}.
\end{equation}
 In particular, we observe that, according to this notation, we have
$$
P=d_{j\alpha k\beta } \qquad {\rm and} \qquad Q=d_{r\gamma s\delta }.
$$
For future reference, we also point out the following general symmetry property
\begin{equation}\label{symmetry-1}
d_{j\alpha k\beta }=-d_{k\alpha j\beta },
\end{equation}
which of course implies
\begin{equation}\label{symmetry-1-bis}
d_{j\alpha j\beta }=0.
\end{equation}
Now we perform a direct computation
\begin{eqnarray}\label{k(P,Q)}
\kappa (P,Q)&=& \kappa (d_{j\alpha k\beta },d_{r\gamma s\delta })\\ \nonumber
 &=&d_{k\alpha r\beta } \, d_{j\gamma s\delta } - \,d_{j\alpha r\beta } \, d_{k\gamma s\delta } -\,d_{k\alpha s\beta } \, d_{j\gamma r\delta } + d_{j\alpha s\beta } \, d_{k\gamma r\delta }. \\\nonumber
\end{eqnarray}
We are now in the right position to prove (a). According to \eqref{general-tau} $f$ is harmonic if and only if
\begin{equation}\label{k(P,Q)-bis}
\kappa(P,Q)=-2 P Q .
\end{equation}
If we set $(j,k)=(r,s)$ in \eqref{k(P,Q)}, then it is easy using \eqref{symmetry-1} and \eqref{symmetry-1-bis} to check that \eqref{k(P,Q)-bis} holds.
Next, if $(\alpha,\beta)=(\gamma,\delta)$ \eqref{k(P,Q)} becomes
$$
\kappa(P,Q)=2(d_{j \alpha  s \beta }\, d_{k \alpha  r \beta }-d_{j \alpha  r \beta }\, d_{k \alpha  s \beta })
$$
which, taking into account the definition \eqref{notation1} and computing, gives again \eqref{k(P,Q)-bis}. In order to prove the converse, a case by case inspection shows that in all the  cases where none of the two conditions $(j,k)=(r,s)$, $(\alpha,\beta)=(\gamma,\delta)$ is satisfied, \eqref{k(P,Q)-bis} does not hold. More precisely, by way of example,
assume that $j=r$, $k\neq s$ $\alpha=\gamma$ and $\beta\neq\delta$. Then the explicit computation gives
\begin{eqnarray}\nonumber
\kappa(P,Q)+2P Q&=& z_{j\alpha }(- z_{j\alpha } z_{k\delta } z_{s \beta }+ z_{j \alpha }  z_{k \beta }  z_{s \delta }-
z_{j \delta }  z_{k \beta }  z_{s \alpha }\\ \nonumber
& & \qquad\qquad +   z_{j \beta }  z_{k \delta }  z_{s \alpha
}+   z_{j \delta }  z_{k \alpha }  z_{s \beta }-   z_{j \beta }  z_{k \alpha }
z_{s \delta })
\end{eqnarray}
which does not vanish identically. The other cases are similar, so ending the proof of (a).

To prove (b) we first introduce the following notation
\begin{equation}\label{notation-2}\nonumber
K_{j\alpha k\beta ,r\gamma s\delta }=\kappa (d_{j\alpha k\beta },d_{r\gamma s\delta }).
\end{equation}
By using \eqref{symmetry-1}, it is easy to deduce the following general symmetry properties
\begin{equation}\label{symmetry-2}
\begin{aligned}
K_{j\alpha k\beta ,r\gamma s\delta }&= -\,K_{k\alpha j\beta ,r\gamma s\delta }\\
K_{j\alpha k\beta ,r\gamma s\delta }&= -\,K_{j\alpha k\beta ,s\gamma r\delta }\\
K_{j\alpha k\beta ,j\gamma k\delta }&=  -2\, d_{j\alpha k\beta } \, d_{j\gamma k\delta }\\
K_{k\gamma s\delta ,r\gamma s\delta }&=  -2\, d_{k\gamma s\delta  } \, d_{r\gamma s\delta }\\
K_{r\alpha k\beta ,r\gamma s\delta }&=  -\, d_{s\alpha r\beta  } \, d_{k\gamma r\delta }
-\, d_{k\alpha r\beta  } \, d_{s\gamma r\delta }
\end{aligned}
\end{equation}

Now, we proceed to the computation of the bitension field $\tau^2(f)$.
Starting from \eqref{general-tau}, it is easy to obtain
\begin{eqnarray}\label{general-bi-tau}
\tau^2(f)&=&-4\tau(f)-2\, \tau(\kappa(P,Q)) \, Q^{-2}\\ \nonumber
& &\qquad\qquad -2\kappa(P,Q)\, \tau(Q^{-2})-4\,\kappa(\kappa(P,Q),Q^{-2}).
\end{eqnarray}
Next, using \eqref{general-tau} and
$$
\kappa(\kappa(P,Q),Q^{-2})= -2\,Q^{-3}\, \kappa(\kappa(P,Q),Q)\,,\quad \tau(Q^{-2})=4(n-4)\,Q^{-2}
$$
we can rewrite \eqref{general-bi-tau} as follows
\begin{eqnarray}\label{general-bi-tau-bis}
Q^{3} \, \tau^2(f)&=&16\, PQ^2 +8(5-n)\,Q\, \kappa(P,Q)\\ \nonumber
& &\qquad\qquad -2\, Q\, \tau(\kappa(P,Q)) +8\,\kappa(\kappa(P,Q),Q).
\end{eqnarray}
Next, we compute explicitly
\begin{eqnarray}\label{k(k(P,Q),Q)}
\kappa(\kappa(P,Q),Q)&=&-d_{k\gamma s\delta } \, K_{j\alpha r\beta ,r\gamma s\delta }+d_{k\gamma r\delta }
\, K_{j\alpha s\beta ,r\gamma s\delta }\\ \nonumber
& & \qquad\qquad -d_{k\alpha s\beta } \, K_{j\gamma r\delta
,r\gamma s\delta }+d_{k\alpha r\beta } \, K_{j\gamma s\delta ,r\gamma s\delta
}\\ \nonumber
& & +d_{j\gamma s\delta } \, K_{k\alpha r\beta ,r\gamma s\delta }-d_{j\gamma r\delta }
\, K_{k\alpha s\beta ,r\gamma s\delta }\\ \nonumber
& &\qquad\qquad +d_{j\alpha s\beta } \, K_{k\gamma r\delta
,r\gamma s\delta }-d_{j\alpha r\beta } \, K_{k\gamma s\delta ,r\gamma s\delta }. \\ \nonumber
\end{eqnarray}
Now, using the various symmetries i.e. \eqref{symmetry-1}, \eqref{symmetry-1-bis}, \eqref{symmetry-2}, after a long computation we find that \eqref{k(k(P,Q),Q)} takes the following form
\begin{eqnarray}\label{k(k(P,Q),Q)-bis}
\kappa(\kappa(P,Q),Q)&=&-\, d_{j\gamma s\delta} \, d_{k\gamma r\delta} \, d_{r\alpha s\beta}+\, d_{j\gamma r\delta}
\, d_{k\gamma s\delta} \, d_{r\alpha s\beta}\\ \nonumber
& &\qquad\qquad -\, d_{j\gamma s\delta} \, d_{k\alpha r\beta}
\, d_{r\gamma s\delta}+2 \, d_{j\gamma r\delta} \, d_{k\alpha s\beta} \, d_{r\gamma s\delta
}\\ \nonumber
& & -2 \, d_{j\alpha s\beta} \, d_{k\gamma r\delta} \, d_{r\gamma s\delta}+\, d_{j\alpha r\beta}
\, d_{k\gamma s\delta} \, d_{r\gamma s\delta}\\ \nonumber
& &\qquad\qquad +\, d_{j\gamma s\delta} \, d_{k\gamma r\delta}
\, d_{s\alpha r\beta}-\, d_{j\gamma r\delta} \, d_{k\gamma s\delta} \, d_{s\alpha r\beta}\\ \nonumber
& &+2
\, d_{j\gamma s\delta} \, d_{k\alpha r\beta} \, d_{s\gamma r\delta}-\, d_{j\gamma r\delta}
\, d_{k\alpha s\beta} \, d_{s\gamma r\delta}\\ \nonumber
& & \qquad\qquad +\, d_{j\alpha s\beta} \, d_{k\gamma r\delta}
\, d_{s\gamma r\delta}-2 \, d_{j\alpha r\beta} \, d_{k\gamma s\delta} \, d_{s\gamma r\delta}. \\ \nonumber
\end{eqnarray}
Next, we compute the other relevant term
\begin{eqnarray} \label{tau(k(P,Q)}
& &\tau (\kappa(P,Q)) \\ \nonumber
&=&  -4(n-1)\, K_{j\alpha k\beta ,r\gamma s\delta } +2\, K_{k\alpha r\beta ,j\gamma s\delta }-2\, K_{j\alpha r\beta ,k\gamma s\delta } \\ \nonumber
& & \qquad\qquad-2\, K_{k\alpha s\beta ,j\gamma r\delta } +2\, K_{j\alpha s\beta ,k\gamma r\delta }  \\ \nonumber
&=& -4(n-1)\, \kappa (P,Q) +2\, K_{k\alpha r\beta ,j\gamma s\delta }-2\, K_{j\alpha r\beta ,k\gamma s\delta }  \\ \nonumber
& & \qquad\qquad-2\, K_{k\alpha s\beta ,j\gamma r\delta } +2\, K_{j\alpha s\beta ,k\gamma r\delta } \\ \nonumber
&=& -4 (n-1) (d_{j \gamma  s \delta }\, d_{k \alpha  r \beta }-d_{j \gamma  r \delta } \,d_{k \alpha
s \beta }+d_{j \alpha  s \beta }\, d_{k \gamma  r \delta }-d_{j \alpha  r \beta }\, d_{k \gamma s \delta })  \\\nonumber
& &\qquad\qquad +2(d_{r \alpha  j \beta } \,d_{k \gamma  s \delta }-d_{j \gamma  s \delta }\, d_{r \alpha k \beta }+d_{j \gamma  k \delta }\, d_{r \alpha  s \beta }-d_{k \gamma  j \delta } \,d_{r \alpha s \beta }\\\nonumber
& & \qquad\qquad  +d_{r \gamma  j \delta } \,d_{k \alpha  s \beta }-d_{j \alpha  s \beta }\, d_{r \gamma
k \delta }+d_{j \alpha  k \beta } \,d_{r \gamma  s \delta }-d_{k \alpha  j \beta }\, d_{r \gamma
s \delta }\\\nonumber
& & \qquad\qquad   -d_{s \alpha  j \beta } \,d_{k \gamma  r \delta }+d_{j \gamma  r \delta } \,d_{s \alpha
k \beta }-d_{j \gamma  k \delta } \,d_{s \alpha  r \beta }+d_{k \gamma  j \delta } \,d_{s \alpha
r \beta }\\\nonumber
& & \qquad\qquad   -d_{s \gamma  j \delta } \,d_{k \alpha  r \beta }+d_{j \alpha  r \beta }\, d_{s \gamma
k \delta }-d_{j \alpha  k \beta } \,d_{s \gamma  r \delta }+d_{k \alpha  j \beta } \,d_{s \gamma
r \delta }).
\end{eqnarray}
We then substitute \eqref{k(k(P,Q),Q)-bis} and \eqref{tau(k(P,Q)}
into \eqref{general-bi-tau-bis}, which becomes, after a long but straightforward  simplification
taking into account the symmetries \eqref{symmetry-1} and \eqref{symmetry-1-bis},  the following expression
\begin{equation}\label{last}
 Q^{3} \, \tau^2(f)=  16 \, d_{s\alpha r\beta}\, (- \, d_{s\gamma j\delta} \, d_{k\gamma r\delta}+ \, d_{j\gamma r\delta } \,  d_{s\gamma k\delta}+ \, d_{k\gamma j\delta } \, d_{s\gamma r\delta }).
\end{equation}
Finally, a simple direct computation, using the definition \eqref{notation1}, shows that the right-hand side of \eqref{last} vanishes.
\end{proof}

\begin{example}
Let $P,Q:\U {n+4}\to\cn$ be the function given by the $2\times 2$ determinants
$$
P(z)=z_{11}z_{22}-z_{12}z_{21}\ \ \text{and}\ \ Q(z)=z_{33}z_{44}-z_{34}z_{43}.
$$
Further let $f(z)=P(z)/Q(z)$ be defined on the open and dense subset
$$
W_Q=\{z\in\U{n+4}|\ Q(z)\neq 0\}.
$$
According to Theorem \ref{theorem-biharmonicity-2x2-determinants}, $f$ gives rise to a proper biharmonic function on the special unitary group $\SU{n+4}$.  This is clearly invariant under the action of the subgroup
$\text{\bf S}(\U 2\times\U 2\times\U n)$ so it induces a proper biharmonic function on the complex flag manifold
$$
\SU{n+4}/\text{\bf S}(\U 2\times\U 2\times\U n)
$$
via the harmonic morphism
$$
\pi:\SU{n+4}\to\SU{n+4}/\text{\bf S}(\U 2\times\U 2\times\U n).
$$
Note that in the particular case when $n=0$ the flag manifold is the complex Grassmannian
$\SU{4}/\text{\bf S}(\U 2\times\U 2)$.
\end{example}

\begin{remark}\label{remark-other-examples} Parts of the computations in the proof of Theorem~\ref{theorem-biharmonicity-2x2-determinants} were carried out and checked by using the software \emph{Mathematica}. Our result provides an ample family of new proper biharmonic functions. However, we point out that, by using \emph{Mathematica}, we were also able to check, in various examples, that the conclusion of the theorem is still true in the case of the quotient of $3 \times 3$ determinants. Similarly, the same happened if $P$ and $Q$ are suitable linear combinations of $2 \times 2$ and $3 \times 3$ determinants, respectively. So we believe that Theorem~\ref{theorem-biharmonicity-2x2-determinants} can reasonably be extended to include more general situations which involve $k \times k$ determinants. Since the amount of computational effort required to handle these cases appear to be extremely heavy, we find it reasonable not to include these developments in the present work.
\end{remark}

\section{The Duality}\label{section-duality}

The approach and the methods of this section were introduced in \cite{Gud-Sve-1}. Let $G$ be a non-compact semisimple Lie group with the Cartan decomposition $\g=\k+\p$ of the Lie algebra of $G$ where $\k$ is the Lie algebra of a maximal compact subgroup $K$. Let $G^\cn$ denote the complexification of $G$ and $U$ be the compact subgroup of $G^\cn$ with Lie algebra $\un=\k+i\p$. Let $G^\cn$ and its subgroups be equipped with a left-invariant semi-Riemannian metric which is a multiple of the Killing form by a negative constant.  Then the subgroup $U$ of $G^\cn$ is Riemannian and $G$ is semi-Riemannian.

Let $f:W\to\cn$ be a real analytic function from an open subset $W$ of $G$. Then $f$ extends uniquely to a holomorphic function $f^\cn:W^\cn\to\cn$ from some open subset $W^\cn$ of $G^\cn$. By restricting this to $U\cap W^\cn$ we obtain a real analytic function $f^*:W^*\to\cn$ on some open subset $W^*$ of $U$. The function $f^*$ is called the {\it dual function} of $f$.

\begin{theorem}[Duality principle]\label{theo-duality}
A complex-valued function $f:W\to\cn$ is proper r-harmonic if and only if its dual $f^*:W^*\to\cn$ is proper r-harmonic.
\end{theorem}
\begin{proof}
It is sufficient to show that
\begin{equation}\label{eq:tau-duality}
\tau(f^{\ast})=-(\tau(f))^{\ast}.
\end{equation}
Let the left-invariant vector fields $X_1,\dots,X_n\in\p$ form a global orthonormal frame for the distribution generated by $\p$
and similarly $Y_1,\dots,Y_m\in\k$ form a global orthonormal frame for the distribution generated by $\k$. Now, we compute separately both sides of \eqref{eq:tau-duality}. According to the semi-Riemannian version of \eqref{tau-kappa-alie-groups}, we have
$$
\tau (f)=-\sum_{k=1}^m Y_k^{2}(f) +\sum_{k=1}^n
X_k^{2}(f).
$$
Next,
\begin{eqnarray}\label{eq-tau-star1}
[\tau (f)]^*&=&-\sum_{k=1}^m [Y_k^{2}(f)]^* +\sum_{k=1}^n[X_k^{2}(f)]^*\nonumber\\
&=&-\sum_{k=1}^m\left( [Y_k^{2}(f)]^{\C}\right)\Bigl\lvert_{U} +
\sum_{k=1}^n \left([X_k^{2}(f)]^{\C}\right)\Bigl\lvert_{U}\nonumber\\
&=&-\sum_{k=1}^m\left( [Y_k(Y_k (f))^{\C})]\right)\Bigl\lvert_{U} +
\sum_{k=1}^n \left([X_k(X_k (f))^{\C})]\right)\Bigl\lvert_{U}\nonumber\\
&=&-\sum_{k=1}^m\left( [Y_k(Y_k (f^{\C}))]\right)\Bigl\lvert_{U} +
\sum_{k=1}^n \left([X_k(X_k (f^{\C}))]\right)\Bigl\lvert_{U}\nonumber\\
&=&-\sum_{k=1}^m\left( Y_k^2(f^{\C})\right)\Bigl\lvert_{U} +
\sum_{k=1}^n \left(X_k^2 (f^{\C})\right)\Bigl\lvert_{U}.
\end{eqnarray}
On the other hand, the holomorphicity of $f^{\C}$ implies
\begin{eqnarray}\label{eq-tau-star2}
\tau (f^*)&=&\sum_{k=1}^m [Y_k^{2}(f^*)]
+\sum_{k=1}^n[(i\,X_k)^{2}(f^*)]\nonumber\\
&=&\sum_{k=1}^m [Y_k^{2}(f^{\C}\Bigl\lvert_{U})]
+\sum_{k=1}^n[(i\,X_k)^{2}(f^{\C}\Bigl\lvert_{U})]\nonumber\\
&=&\sum_{k=1}^m\left( Y_k^2(f^{\C})\right)\Bigl\lvert_{U}
-\sum_{k=1}^n \left(X_k^2 (f^{\C})\right)\Bigl\lvert_{U}.
\end{eqnarray}
Finally, comparing \eqref{eq-tau-star1} and \eqref{eq-tau-star2} we obtain \eqref{eq:tau-duality}.
\end{proof}

\begin{remark}\label{re:duality}
We point out that a function $f:W\to\cn$ is $K$-invariant if and only if its dual $f^*:W^*\to\cn$ is $K$-invariant. In particular, the duality principle of Theorem~\ref{theo-duality} is valid for the corresponding functions on the quotient spaces.
\end{remark}

\begin{example}
Let $f^*:W^*\subset\SO{4} \to\cn$ be the complex-valued function defined
locally on $\SO{4}$ by
$$
f^*:x\mapsto\frac{p_1 x_{11}+p_2 x_{21}+p_3 x_{31}+p_4 x_{41}}{q_1 x_{11}+q_2 x_{21}+q_3 x_{31}+q_4 x_{41}},
$$
where $p,q\in\cn^4$ are linearly independent, $(q,q)=0$ and $(p,q)\neq 0$.
According to Theorem~\ref{teo-quotient-linear}, we know that $f^*:W^*\to\cn$ is proper biharmonic. According to Example 7.2 in \cite{Gud-Sve-1}, the dual function $f:W\subset\SOO 13\to\cn$  is given by
$$f:x\mapsto\frac{p_1 x_{11}+p_2 (-i x_{21})+p_3 (-i x_{31})+p_4 (-i x_{41})}{q_1 x_{11}+q_2 (-i x_{21})+q_3 (-i x_{31})+q_4 (-i x_{41})}.
$$
The duality principle of Theorem~\ref{theo-duality} tells us that $f$ is also proper biharmonic. The functions $f$ and $f^*$ are both
$\SO 3$-invariant so, by Remark~\ref{re:duality}, they induce biharmonic functions $h:U\subset\H^3\to\cn$ and  $h^*:U^*\subset\s^3\to\cn$ on the quotient spaces of left cosets
$$\H^3=\SOO 13/\SO 3\ \ \text{and}\ \ \s^3=\SO{4}/\SO 3.$$ It is worth to point out that, by means of this construction, we produced an equivalent description of the proper biharmonic functions  given in Examples~\ref{example-proper-biharmonic-S^3} and~\ref{example-hyperbolic} and obtained a geometric relationship between the two families of examples.

\end{example}

\appendix

\section{}\label{general-F(z)}

The complex general linear group $\GLC n$ is an open subset of $\cn^{n\times n}\cong\cn^{n^2}$ and inherits its standard complex structure.  As before, let $z_{ij}$ with $1\leq i,\,j \leq n$ denote the matrix coefficients of a generic element $z\in\GLC n$.
Let $G$ be a Lie subgroup of $\GLC n$, $F:U\to\cn$ be a holomorphic function defined on an open subset of $\GLC n$ and $f:U\cap G\to\cn$ be the restriction of $F$ to $G$.

Following Lemma 3.2 of \cite{Gud-5}, we can use \eqref{eq:derivativeZ}, \eqref{eq:derivativeZZ} and \eqref{tau-kappa-alie-groups} to compute the tension field of $f$
\begin{equation}\label{tension-general-F}
\tau(f)= \sum_{1\leq i,j,k,\ell \leq n} \frac {\partial ^2 f}{\partial z_{ij}\partial z_{k\ell}}\,\, \kappa(z_{ij},z_{k\ell}) \,+\,
\sum_{1\leq i,j\leq n} \frac {\partial  f}{\partial z_{ij} } \, \, \tau (z_{ij}) \,\, .
\end{equation}

\begin{example}
For the unitary group $\U n$ we can utilize Lemma~\ref{lemm:complex} to make formula \eqref{tension-general-F} more explicit
\begin{equation*}\label{tension-general-F-su-U(n)}
\tau(f)= - \sum_{1\leq i,j,k,\ell \leq n} \frac {\partial ^2 f}{\partial z_{ij}\partial z_{k\ell}}\, z_{kj}\, z_{i\ell}\, -n
\sum_{1\leq i,j\leq n} \frac {\partial  f}{\partial z_{ij} } \, z_{ij}\, .
\end{equation*}
It can be shown that in this case the bitension field satisfies
\begin{eqnarray*}
\tau^2(f)&=&-\,\, \sum_{i,j,k,\ell,\alpha,\beta,\gamma,\delta=1}^n\ \frac{\partial ^4 f}{\partial z_{ij}\partial z_{k\ell}\partial z_{\alpha \beta}\partial z_{\gamma \delta}} \,\,\,z_{i\ell} z_{kj} z_{\alpha \delta} z_{\gamma \beta}  \\ \nonumber
& &-\sum_{i,j,k,\ell,\alpha,\beta=1}^n\ \frac{\partial ^3 f}{\partial z_{ij}\partial z_{k\ell}\partial z_{\alpha \beta}} \, [2 z_{kj} z_{\alpha \ell} z_{i \beta}+2 z_{i \ell} z_{\alpha j} z_{k \beta}\\
& &\qquad\qquad\qquad\qquad\qquad\qquad +\ n\, z_{ij} z_{\alpha \ell} z_{k \beta}+n\, z_{i \ell} z_{kj} z_{\alpha \beta}]\\ \nonumber
& &-\sum_{i,j,k,\ell=1}^n \ \frac{\partial ^2 f}{\partial z_{ij}\partial z_{k\ell}} \,\, [(n^2+2) z_{ij} z_{k \ell} +4n\, z_{i \ell} z_{k j} ] \\ \nonumber
& &- \ n^2 \,\,\sum_{i,j=1}^n \ \frac{\partial f}{\partial z_{ij}} \, \, z_{ij}  \, . \\ \nonumber
\end{eqnarray*}
\end{example}

\begin{example}
For the special orthogonal group $\SO n$, Lemma~\ref{lemm:real} gives to the following version of \eqref{tension-general-F}
\begin{eqnarray*}\label{tension-general-F-su-O(n)}
\tau(f)&=& - \frac{1}{2} \, \sum_{1\leq i,j,k,\ell \leq n} \frac {\partial ^2 f}{\partial x_{ij}\partial x_{k\ell}}\,x_{kj}\, x_{i\ell}\\
& &\qquad\qquad\qquad +\frac{1}{2}\, \sum_{1\leq i,j\leq n}\,\frac {\partial ^2 f}{\partial x_{ij}^2}\, \, -\frac{(n-1)}{2}
\sum_{1\leq i,j\leq n} \frac {\partial  f}{\partial x_{ij} } \, x_{ij} \,\, .
\end{eqnarray*}
\end{example}

It is our hope that the ideas presented in this appendix will be helpful for carrying out computer calculations in the spirit of Remark~\ref{remark-other-examples} above.

\end{document}